\documentclass[12pt]{article}
\usepackage{dsfont}
\usepackage{amsmath,amssymb}
\usepackage{graphicx}

\newtheorem{lemma}{Lemma}
\newtheorem{proof}{Proof}
\newtheorem{theorem}{Theorem}
\newtheorem{corollary}{Corollary}
\newtheorem{remark}{Remark}
\newtheorem{example}{Example}

\begin{document}

\vskip 3mm

\noindent  \textbf{Strongly mixed random errors in Mann's iteration algorithm for a contractive real function}
\vskip 3mm

\vskip 5mm
\noindent Hassina ARROUDJ$^{1}$, Idir ARAB$^{2}$ and Abdelnasser DAHMANI$^{3}$

\noindent $^{1}$Laboratoire de Math\'{e}matiques Appliqu\'{e}es, Facult\'{e} des Sciences Exactes, Universit\'{e} A. Mira Bejaia, Algerie.\\
\noindent $^{2}$CMUC, Department of Mathematics, University of Coimbra, Portugal\\
\noindent $^{3}$Centre Universitaire de Tamanrasset.

\noindent $^{1}$e-mail: has\_arroudj@hotmail.com\\
 $^{2}$ idir@mat.uc.pt\\
$^{3}$a\_dahmany@yahoo.fr

\vskip 3mm
\noindent \textbf{Keywords:} Fixed point; Fuk-Nagaev's inequalities; Mann's iteration algorithm; $\alpha $-mixing; Rate of
convergence; confidence domain.
\vskip 3mm

\noindent \textbf{Abstract}

This work deals with the Mann's stochastic iteration algorithm under $\alpha -$%
mixing random errors. We establish the Fuk-Nagaev's inequalities that enable us
to prove the almost complete convergence with its corresponding rate of convergence.
Moreover, these inequalities give us the possibility of constructing a confidence
interval for the unique fixed point. Finally, to check the feasibility and validity
of our theoretical results, we consider some numerical examples, namely a classical example from astronomy.
\vskip 4mm

\noindent 1.   \textbf{Introduction}

In many mathematical problems arising from various domains, the existence of a solution is the same as the existence of a fixed point by some appropriate transformation of the problem. The most known problem in that framework is the root existence which can be tackled easily as the existence of a fixed point and vice versa. Therefore, the fixed point theory is of paramount importance in engineering sciences and many areas of mathematics. Fixed point theory provides conditions under which maps have the existence and uniqueness of solutions. Over the last decades, that theory has been revealed as one of the most significant tool in the study of nonlinear problems. In particular, in many fields, equilibria or stability are fundamental concepts that can be described in terms of fixed points. For example, in economics, a Nash equilibrium of a game is a fixed point of the game's best response correspondence. However, in informatics, programming language compilers use fixed point computations for program analysis, for example in data-flow analysis, which is often required for code optimization. The vector of PageRank values of all web pages is the fixed point of a linear transformation derived from the World Wide Web's link structure. In astronomy, the eccentric anomaly $E$ of a given planet is related to a fixed point equation that cannot not be solved analytically, this will be well described in example (\ref{astronomy_example}) and many examples could be found in engineering sciences such as physics, geology, chemistry, biology, mechanical statistics, etc.\\
Mathematically, a fixed point problem is presented under the following form

\begin{equation}\label{Fixed-Point}
\text{Find }x\in X\text{ such that }Fx=x
\end{equation}%
Where $F$ is an operator, defined on a space $X$. The solutions of that equation if they exist are called "fixed points" of the mapping $F$. The classical result in fixed point theory is the Banach fixed-point theorem \cite{banac}; it ensures the existence and uniqueness of a fixed point of certain self-maps of a metric space. Additionally, it provides a constructive numerical method to approximate the fixed point.

After verifying the existence and uniqueness conditions, it is necessary to find (or approximate) the unique fixed point of the problem (\ref{Fixed-Point}). This leads to find the unique root of $F-id_X$ where $id_X$ denotes the identity operator on $X$. Analytically, to find that root, one has to reverse the operator $F-id_X$, and one could immediately think about the difficulty when dealing with inversion and most of the time that task is impossible. Alternatively, numerical methods become the most appropriate and have attracted many researches these last decades. The pioneering work after Picard's iterative method was introduced by Mann \cite{Man} to remedy the problem of convergence while using the Picard's method for approximating the fixed point of nonexpansive mapping. Later, many modified algorithms were introduced, by considering the stochastic part, i.e., considering the errors generated by the  numerical evaluation of the algorithm. For an account of relevant literature on that topic, see (\cite{Ber,Ceg,cha2,Hus,Ish,Kan,Kim,Liu3}).

In the framework of this paper, we consider the Mann iterative algorithm as described
in (\ref{smi}) by taking into account the committed errors at each evaluation of the
approximated fixed point $x_n$. These errors are supposed to be random and modeled
by random variables and we suppose them to be strong mixing. Recall that a sequence $\left( \xi _{i}\right) $
is said to be strong mixing or $\alpha -$mixing if the following condition is satisfied:%
\begin{equation}
\alpha \left( n\right) =\sup_{A\in \mathcal{F}_{-\infty }^{k},B\in \mathcal{F%
}_{k+n}^{+\infty }}\left\vert \mathbb{P}\left( A\cap B\right) -\mathbb{P}%
\left( A\right) \mathbb{P}\left( B\right) \right\vert \underset{%
n\longrightarrow +\infty }{\longrightarrow }0  \label{mf}
\end{equation}%
where $\mathcal{F}_{l}^{m}$ denotes the $\sigma $-algebra engendered by events of
the form $\left\{ \left( \xi _{i_{1}},\cdots ,\xi _{i_{k}}\right) \in
E\right\} $, where $l\leq i_{1}<i_{2}<\cdots <i_{k}\leq m$ and $E$ is a
Borel set.

The notion of $\alpha $-mixing was firstly introduced by Rosenbaltt in
1956 \cite{Ros} and the central limit theorem has been established.
The strong mixing random variables have many interest in linear
processes and found many application in finance, for more examples
and properties concerning the mixing notions, see \cite{Dou,Ibr}.

In this paper, we establish Fuk-Nagaev inequalities. These inequalities allow us to prove the almost complete
convergence of Mann's algorithm to the fixed point, with convergence rate and the possibility of giving the a confidence interval.
To strengthen the obtained theoretical results, some numerical examples are considered.\newline
The rest of the paper is organized as follows: In section 2 the statement of the problem is described and some known results are recalled. In section 3, some new results were
established by using stochastic methods. In section 4, the validity of our
approach is checked up by considering some numerical examples.\\

\noindent 2. \textbf{Preliminaries and known results} \\

Let $\left( \Omega ,\mathcal{F},\mathbb{P}\right) $ be a probability space
and $f:\mathbb{R\rightarrow R}$ \ a non-linear function. We consider the stochastic Mann's iteration algorithm:%
\begin{equation}
x_{n+1}=\left( 1-a_{n}\right) x_{n}+b_{n}f\left( x_{n}\right) +c_{n}\xi _{n},
\label{smi}
\end{equation}%
where the sequences of positive numbers $(a_{n})_{n\geq 1}$ and $(b_{n})_{n\geq 1}$  satisfy the following conditions%
\begin{eqnarray*}
\lim_{n\rightarrow +\infty}b_{n}=\lim_{n\rightarrow +\infty}a_{n}&=&0\text{ and }\sum_{n=1}^{+\infty
}b_{n}=\sum_{n=1}^{+\infty}a_{n}=+\infty \\
\sum_{n=1}^{+\infty }c_{n}^{2} &<&+\infty .
\end{eqnarray*}%

Without loss of generality, we take
\begin{equation*}
a_{n}=b_{n}=\frac{a}{n}\ \text{ and } \ c_{n}=\frac{a}{n^{2}}, a>0.
\end{equation*}

Hence, the stochastic Mann's iteration algorithm (\ref{smi}) takes the following form:
\begin{equation}
x_{n+1}=\left( 1-\frac{a}{n}\right) x_{n}+\frac{a}{n}\left[ f\left(
x_{n}\right) +\frac{1}{n}\xi _{n}\right] .  \label{mannstocha}
\end{equation}
We now introduce some classical hypothesis that will be useful tools for the proof of established results in the sequel:

(H1) : The fixed point $x^{\ast }$ satisfies
\begin{equation}
\exists \ N>0,\ \left\vert x_{1}-x^{\ast }\right\vert \leq N<+\infty .  \label{apriori}
\end{equation}

(H2) : The function $f$ is contractive, i.e, it satisfies the following property:%
\begin{equation}
\forall \ x,y\in \mathbb{R},\left\vert f\left( x\right) -f\left( y\right)
\right\vert \leq c\left\vert x-y\right\vert ,c\in \left( 0,1\right) .
\label{contraction}
\end{equation}

(H3) : The random variables $(\xi _{i})$ fulfill the condition of
uniform decrease, that is%
\begin{equation}
\exists \ p>2,%
\begin{array}{c}
\end{array}%
\ \forall \ t>0,\ \mathbb{P}\left\{ \left\vert \xi _{i}\right\vert >t\right\} \leq
t^{-p}.  \label{queues}
\end{equation}

(H4) : The coefficients of the $\alpha $-mixing sequence $\left( \xi
_{n}\right) _{n}$ satisfy the following arithmetic decay condition:%
\begin{equation}
\exists \ d\geq 1,\ \exists \ \beta >1:\alpha \left( n\right) \leq d \ n^{-\beta
},\forall \ n\in \mathbb{N}^{\ast }.  \label{cm}
\end{equation}

(H5) : The $\alpha $-mixing coefficients satisfy the following condition
\begin{equation}
\exists \ \rho >0,\ \rho \frac{\left( \beta +1\right) p}{\beta +p}>2.
\label{dec}
\end{equation}


\begin{remark}
The assumption (H1) is classical. Arbitrary choice of $x_{1}$ and the
existence of $x^{\ast }$ allows us to assume such supposition. The contraction's assumption
(H2) ensures the existence and uniqueness of the fixed point $x^{*}$ of the function $f$ according to Banach's
theorem for fixed point. When the function $f$ is derivable, the condition (H2) is equivalent to the boundness of the derivative $f^{\prime}$, i.e, $%
\exists \ c >0, \sup\limits_{x}\left\vert f^{\prime }\left( x\right) \right\vert
\leq c<1$. The hypothesis (H3) is satisfied for all bounded random variables and Gaussian ones.
Assumption (H4) is used in order to characterize the dependence structure of errors.
Moreover, combined to (H3), the assumption (H4) allows the obtention of Fuk-Nagaev's inequalities \cite{Rio}, which ensures the almost complete convergence result. As a particular example, the geometric $\alpha $-mixing sequence $\left( \xi _{i}\right) _{i}$ and its mixing coefficients are defined as follows
\begin{equation*}
\exists \ d_{0}>0,\exists \ \kappa \in \left( 0,1\right) :\alpha \left( n\right)
\leq d_{0}\ \kappa ^{n},\forall \ n\ \in \mathbb{N}^{\ast }.
\end{equation*}

The assumption (H5) will be useful for specifying the rate of almost complete convergence of the stochastic Mann's iteration algorithm. That condition is classical, see \cite{Ait,Fer}.
\end{remark}
First, we state the following theorem which will be used in the sequel during the proof of the main result.
\begin{theorem}
\label{Fuk-Nagaev}Let $\left( \xi _{i}\right) _{i\in \mathbb{N}^{\ast }}$ be
a \ centered sequence of real-valued random variables and $\left( \alpha
_{n}\right) _{n\in \mathbb{N}^{\ast }}$ the corresponding sequence of mixing coefficients
as defined in (\ref{mf}) such  that the hypothesis (H2) and (H4)  are satisfied.
Let us set
\begin{equation*}
s_{n}^{2}=\sum\limits_{i=1}^{n}\sum\limits_{k=1}^{n}\left\vert
Cov\left(\xi_{i},\xi_{k}\right) \right\vert .
\end{equation*}%
Then, for every real numbers $r\geq 1$ and $\lambda >0,$ we have%
\begin{equation*}
\mathbb{P}\left\{ \sup_{k\in \left[ 1,n\right] }\left\vert
\sum_{i=1}^{k}\xi _{i} \right\vert \geq
4\lambda \right\} \leq 4\left( 1+\frac{\lambda ^{2}}{rs_{n}^{2}}\right) ^{%
\frac{-r}{2}}+2cnr^{-1}\left( \frac{2r}{\lambda }\right) ^{\frac{\left(
\beta +1\right) p}{\left( \beta +p\right) }}.
\end{equation*}
\end{theorem}

\begin{proof} The proof is well detailed in \cite{Rio}, page 84 to 87.
\end{proof}

\begin{lemma}
Using the hypothesis (H1), we get the following inequality:
\begin{equation}
\left\vert x_{n+1}-x^{\ast }\right\vert \leq N\prod\limits_{i=1}^{n}\left(
1-\frac{a\left( 1-c\right) }{i}\right) +\sum\limits_{i=1}^{n}\frac{a}{i^{2}}%
\prod\limits_{j=i+1}^{n}\left( 1-\frac{a\left( 1-c\right) }{j}\right)
\left\vert \xi _{i}\right\vert .  \label{formule}
\end{equation}
\end{lemma}

\begin{proof}
The proof is straightforward by induction on $n$.%
\end{proof}

\begin{lemma}
For all constants $a<1$, we have the following inequalities%
\begin{equation}
\prod\limits_{j=i+1}^{n}\left( 1-\frac{a\left( 1-c\right) }{j}\right) \leq
\left( \frac{i+1}{n+1}\right) ^{a\left( 1-c\right) },  \label{produit}
\end{equation}%
and,
\begin{equation}
\sum_{i=1}^{n}\frac{a}{i^{2}}\prod\limits_{j=i+1}^{n}\left( 1-\frac{a\left(
1-c\right) }{j}\right) \leq \frac{aS}{\left( n+1\right) ^{a\left( 1-c\right)
}}.  \label{somprod}
\end{equation}
\end{lemma}

\begin{proof}
We have,
\begin{equation*}
\prod\limits_{j=i+1}^{n}\left( 1-\frac{a\left( 1-c\right) }{j}\right) \leq
\exp \left( -a\left( 1-c\right) \sum_{j=i+1}^{n}\frac{1}{j}\right) \leq
\left( \frac{i+1}{n+1}\right) ^{a\left( 1-c\right) }.
\end{equation*}%
The second inequality follows immediately  from inequality (\ref{produit}) by setting $S$ the sum of the
convergent series $\frac{\left( i+1\right) }{i^{2}}^{a\left(1-c\right) }$.
\end{proof}

\vskip 3mm

\noindent 3.  \textbf{Convergence of Mann iterative algorithm}

\begin{theorem}
Under the assumptions (H1)--(H5), we have for any real positive $\rho $ such that $\frac{2(\beta+p)}{p(\beta+1)}<\rho <a\left( 1-c\right) <1$, we have:
\begin{equation}
x_{n+1}-x^{\ast }=\mathcal{O}\left(\frac{\sqrt{\ln n}}{n^{a\left( c-1\right) -\rho }}\right)\ \ \
\ a.co.\label{Rate}
\end{equation}
\end{theorem}

\begin{proof}
Using the inequality (\ref{formule}), we have
\begin{eqnarray}
&&\mathbb{P}\left\{ \left\vert x_{n+1}-x^{\ast }\right\vert >\varepsilon
\right\}  \notag \\
&\leq &\mathbb{P}\left\{ N\prod\limits_{i=1}^{n}\left( 1-\frac{a\left(
1-c\right) }{i}\right) +\sum\limits_{i=1}^{n}\frac{a}{i^{2}}%
\prod\limits_{j=i+1}^{n}\left( 1-\frac{a\left( 1-c\right) }{j}\right)
\left\vert \xi _{i}\right\vert >\varepsilon \right\}  \notag \\
&\leq &\mathbb{P}\left\{ N\prod\limits_{i=1}^{n}\left( 1-\frac{a\left(
1-c\right) }{i}\right) +\sum\limits_{i=1}^{n}\frac{a}{i^{2}}%
\prod\limits_{j=i+1}^{n}\left( 1-\frac{a\left( 1-c\right) }{j}\right)
\mathbb{E}\left\vert \xi _{i}\right\vert \geq \frac{\varepsilon }{2}\right\}
\notag \\
&&+\mathbb{P}\left\{ \sum\limits_{i=1}^{n}\frac{a}{i^{2}}%
\prod\limits_{j=i+1}^{n}\left( 1-\frac{a\left( 1-c\right) }{j}\right)
\left( \left\vert \xi _{i}\right\vert -\mathbb{E}\left\vert \xi
_{i}\right\vert \right) >\frac{\varepsilon }{2}\right\} .  \label{sommation}
\end{eqnarray}

Firstly, we have%
\begin{equation}
\mathbb{P}\left\{ N\prod\limits_{i=1}^{n}\left( 1-\frac{a\left( 1-c\right)
}{i}\right) +\sum\limits_{i=1}^{n}\frac{a}{i^{2}}\prod\limits_{j=i+1}^{n}%
\left( 1-\frac{a\left( 1-c\right) }{j}\right) \mathbb{E}\left\vert \xi
_{i}\right\vert >\frac{\varepsilon }{2}\right\} \leq K_{1}e^{-n^{2a\left(
1-c\right) }\varepsilon ^{2}}.  \label{terme1}
\end{equation}

We set
\begin{equation*}
Z_{i}=\frac{an^{a\left( 1-c\right) }}{i^{2}}\prod\limits_{j=i+1}^{n}\left(
1-\frac{a\left( 1-c\right) }{j}\right) \left( \left\vert \xi _{i}\right\vert
-\mathbb{E}\left\vert \xi _{i}\right\vert \right) .
\end{equation*}%
Note that the random variables $(Z_{i})$ are centered and according to (\ref%
{queues}), we show that, there exists a positive constant $M$ such that,%
\begin{equation}
\forall \ t>0%
\begin{array}{c}
,%
\end{array}%
\mathbb{P}\left\{ \left\vert Z_{i}\right\vert >t\right\} \leq Mt^{-p}.
\label{queuesZ}
\end{equation}%
Finally, we notice that if the random errors $(\xi _{i})$ are $\alpha $-mixing,
then the random variables $(Z_{i})$ remain also with mixing
coefficients less than or equal to those of the sequence $\left( \xi
_{i}\right) _{i}$. Thus, applying the Fuk-Nagaev's exponential inequality given
by Rio (Theorem \ref{Fuk-Nagaev}) to the variables $(Z_{i})$, we obtain for any $\varepsilon >0$ and $r\geq 1:$%
\begin{eqnarray}
\mathbb{P}\left\{ \left\vert x_{n+1}-x^{\ast }\right\vert >\varepsilon
\right\} &\leq &K_{1}e^{-n^{2a\left( 1-c\right) }\varepsilon ^{2}}+4\left( 1+%
\frac{\varepsilon ^{2}n^{2a\left( 1-c\right) }}{4rs_{n}^{2}}\right) ^{\frac{%
-r}{2}}  \notag \\
&&+4Cnr^{-1}\left( \frac{2r}{\varepsilon n^{a\left( 1-c\right) }}\right) ^{%
\frac{\left( \beta +1\right) p}{\left( \beta +p\right) }}  \label{Nagaev}
\end{eqnarray}%
where,
\begin{equation*}
C=2Mp\left( 2p-1\right) ^{-1}\left( 2^{\beta }d\right) ^{\frac{p-1}{\beta +p}%
}\text{ and }s_{n}^{2}=\sum\limits_{i=1}^{n}\sum\limits_{k=1}^{n}\left\vert
Cov\left( Z_{i},Z_{k}\right) \right\vert .
\end{equation*}

Let us bound the double sum of covariances $s_{n}^{2}$, we have:%
\begin{equation*}
s_{n}^{2}=\sum\limits_{i=1}^{n}\sum\limits_{k=1}^{n}\left\vert Cov%
\left( Z_{i},Z_{k}\right) \right\vert =\sum\limits_{i=1}^{n} Var%
\left( Z_{i}\right) +\sum\limits_{i=1}^{n}\sum\limits_{k\neq i}^{n}\left\vert
Cov\left( Z_{i},Z_{k}\right) \right\vert .
\end{equation*}

We recall that
\begin{equation}
\sum\limits_{i=1}^{n} Var\left( Z_{i}\right) \leq
\sum\limits_{i=1}^{n}\frac{a^{2}\left( i+1\right) ^{2a\left( 1-c\right) }}{%
i^{4}} Var\left( \left\vert \xi _{i}\right\vert \right) \leq S_{v}
\label{var}
\end{equation}%
since it is a partial sum of a convergent series with positive terms.

On the other hand, for $i\neq k$, we have
\begin{equation}
\left\vert Cov\left( Z_{i},Z_{k}\right) \right\vert \leq \frac{%
a^{2}\left( i+1\right) ^{a\left( 1-c\right) }}{i^{2}}\frac{\left( k+1\right)
^{a\left( 1-c\right) }}{k^{2}}\left \vert \mathbb{E}\left( \left\vert \xi
_{i}\right\vert -\mathbb{E}\left\vert \xi _{i}\right\vert \right) \left(
\left\vert \xi _{k}\right\vert -\mathbb{E}\left\vert \xi _{k}\right\vert
\right) \right \vert .
\end{equation}%
According to the inequality given by Ibragimov \cite{Ibr} (Theorem 17.2.2
page 307), we obtain:%
\begin{equation}
\left\vert \mathbb{E}\left( \left\vert \xi _{i}\right\vert -\mathbb{E}%
\left\vert \xi _{i}\right\vert \right) \left( \left\vert \xi _{k}\right\vert
-\mathbb{E}\left\vert \xi _{k}\right\vert \right) \right\vert \leq \left(
4+6C\right) \left( \alpha \left( \left\vert i-k\right\vert \right) \right) ^{%
\frac{p-2}{p}},
\end{equation}%
consequently,
\begin{equation}
\left\vert Cov\left( Z_{i},Z_{k}\right) \right\vert \leq
a^{2}\left( 4+6C\right) \frac{\left( i+1\right) ^{a\left( 1-c\right) }}{i^{2}%
}\frac{\left( k+1\right) ^{a\left( 1-c\right) }}{k^{2}}\left( \alpha \left(
\left\vert i-k\right\vert \right) \right) ^{\frac{p-2}{p}}.  \label{ras1}
\end{equation}%
Since the mixing coefficients of the sequence $\left( \left\vert \xi
_{i}\right\vert -\mathbb{E}\left\vert \xi _{i}\right\vert \right) _{i}$ are
less than or equal to those of the sequence $\left( \xi _{i}\right) _{i}$, we get

\begin{eqnarray}
&&\ \sum\limits_{i=1}^{n}\sum\limits_{k\neq i}^{n}\left\vert Cov\left(
Z_{i},Z_{k}\right) )\right\vert \leq
\sum\limits_{i=1}^{n}\sum\limits_{\left\vert i-k\right\vert \leq
u_{n}}a^{2}\left( 4+6C\right) \frac{\left( i+1\right) ^{a\left( 1-c\right) }%
}{i^{2}}\frac{\left( k+1\right) ^{a\left( 1-c\right) }}{k^{2}}\left( \alpha
\left( \left\vert i-k\right\vert \right) \right) ^{\frac{p-2}{p}}  \notag \\
&&+\sum\limits_{i=1}^{n}\sum\limits_{\left\vert i-k\right\vert
>u_{n}}a^{2}\left( 4+6C\right) \frac{\left( i+1\right) ^{a\left( 1-c\right) }%
}{i^{2}}\frac{\left( k+1\right) ^{a\left( 1-c\right) }}{k^{2}}\left( \alpha
\left( \left\vert i-k\right\vert \right) \right) ^{\frac{p-2}{p}}\leq S_{c}.
\label{cov}
\end{eqnarray}%
Combining (\ref{var}) and (\ref{cov}), we obtain:
\begin{equation}
s_{n}^{2}\leq S_{v}+S_{c}=S.  \label{sn2}
\end{equation}%
So, from (\ref{sn2}), we have the inequality
\begin{equation}
\mathbb{P}\left\{ \left\vert x_{n+1}-x^{\ast }\right\vert >\varepsilon
\right\} \leq T_{1}+T_{2}+T_{3}
\end{equation}%
where
\begin{equation*}
T_{1}=K_{1}e^{-n^{2a\left( 1-c\right) }\varepsilon ^{2}},\ T_{2}=4\left( 1+%
\frac{n^{2a\left( 1-c\right) }\varepsilon ^{2}}{4rS}\right) ^{\frac{-r}{2}}%
\text{ and }T_{3}=4Cnr^{-1}\left( \frac{r}{n^{a\left( 1-c\right)
}\varepsilon }\right) ^{\frac{\left( \beta +1\right) p}{\beta +p}}.
\end{equation*}

For a well chosen positive number $r$ and $\varepsilon$, the quantities $T_{1},T_{2}$ and $T_{3}$ become a general terms
of convergent series. Consequently, we obtain
\begin{equation*}
\sum_{n=1}^{+\infty }\mathbb{P}\left\{ \left\vert x_{n+1}-x^{\ast }\right\vert
>\varepsilon \right\} <+\infty
\end{equation*}%
that ensures the almost complete convergence of $\left( x_{n}\right) _{n}$
to the unique fixed point $x^{\ast }.$ The choice of the tuning positive number $r$ will be specified while deriving the corresponding rate of convergence.

Recall that $x_{n}-x^{\ast }=\mathcal{O}\left( \epsilon_{n}\right) $ almost completely $%
(a.co),$ where $\left( \epsilon_{n}\right) _{n}$ is a sequence of real positive
numbers tending to zero, if there exists a positive constant $k$ such that%
\begin{equation*}
\sum_{n=1}^{+\infty }\mathbb{P}\left\{ \left\vert x_{n}-x^{\ast }\right\vert
>k\epsilon_{n}\right\} <+\infty .
\end{equation*}%
Basing on the inequalities obtained above, we take:%
\begin{equation*}
\varepsilon=\varepsilon_{n}=k\epsilon _{n}, \text{ where } k=\sqrt{1+\delta },\text{ }\delta >0\text{ and }\epsilon_{n}=\frac{%
\sqrt{\ln n}}{n^{a\left( 1-c\right) -\rho }}.
\end{equation*}%
Hence, we obtain
\begin{equation}
T_{1}=K_{1}e^{-\left( n+1\right) ^{2a\left( 1-c\right) }\varepsilon
^{2}}\leq K_{1}e^{-\left( 1+\delta \right) \ln n}=\frac{K_{1}}{n^{1+\delta }}%
.  \label{T1}
\end{equation}%
For a suitably chosen $r$ such that $r>\frac{2}{\rho }$, we obtain%
\begin{equation}
T_{2}=4\left( 1+\frac{\left( 1+\delta \right) n^{\rho }}{rS}\right) ^{\frac{%
-r}{2}}\leq K_{2}n^{-\rho \frac{r}{2}}  \label{T2}
\end{equation}%
where
\begin{equation*}
K_{2}=\left( \frac{rS}{1+\sigma }\right) ^{r/2}.
\end{equation*}%
With regard to $T_{3}$, we have%
\begin{eqnarray}
T_{3} &\leq &4Cnr^{-1}\left( \frac{r}{\sqrt{1+\delta }n^{\rho }\ln n}\right)
^{\frac{\left( \beta +1\right) p}{\beta +p}}  \notag \\
&=&4Cr^{\frac{\left( \beta +1\right) p}{\beta +p}-1}\frac{n}{\left( \sqrt{%
1+\delta }n^{\rho }\ln n\right) ^{\frac{\left( \beta +1\right) p}{\beta +p}}}%
  \notag
\end{eqnarray}%
With $r$ chosen as in (\ref{T2}), we deduce
\begin{equation}
T_{3}\leq K_{3}\frac{1}{n^{\rho \frac{\left( \beta +1\right) p}{\beta +p}%
-1}\left( \ln n\right) ^{\frac{\left( \beta +1\right) p}{\beta +p}}}
\label{T3}
\end{equation}%
which is a general term of Bertrand series, it is convergent because of the
hypothesis (H5).
It leads that:
\begin{equation}
\mathbb{P}\left\{ \left\vert x_{n+1}-x^{\ast }\right\vert >\sqrt{1+\delta }%
\frac{\sqrt{\ln n}}{n^{a\left( 1-c\right) -\rho }}\right\} \leq \frac{K_{1}}{%
n^{1+\delta }}+\frac{K_{2}}{n^{\rho \frac{r}{2}}}+\frac{K_{3}}{n^{\rho \frac{%
\left( \beta +1\right) p}{\beta +p}}\left( \ln n\right) ^{\frac{\left( \beta
+1\right) p}{\beta +p}}}.  \label{dering}
\end{equation}%
The right-hand side of the last inequality is a term of a convergent series.
\end{proof}
\begin{remark}
Remark that in the obtained rate of convergence given by the formula
(\ref{Rate}), more the quantity $\frac{2(\beta+p)}{p(\beta+1)}$ is
small, more we have the choice of taking $\rho$ small and consequently
the rate of convergence becomes more interesting.
\end{remark}
\begin{corollary}
Under the assumptions (H1)--(H5), for a given level $\sigma $, there exists a
natural integer $n_{\sigma }$ for which the fixed point $x^{\ast }$ of the function $f$
belongs to closed interval of center $x_{n_{\sigma}}$ and radius $%
\varepsilon $ with a probability greater than or equal to $1-\sigma $.
\begin{equation}
\forall\ \varepsilon >0,\forall \ \sigma >0,\exists \ n_{\sigma }\in \mathbb{N}:%
\mathbb{P}\left\{ \left\vert x_{n_{\sigma }}-x^{\ast }\right\vert \leq
\varepsilon \right\} \geq 1-\sigma .  \label{ic}
\end{equation}
\end{corollary}

\begin{proof}
Indeed, using Kronecker's Lemma, we obtain $\lim\limits_{n\rightarrow
+\infty }\alpha \left( n\right) =0$ which implies
\begin{equation}
\lim_{n\rightarrow +\infty }\frac{K_{1}}{n^{1+\delta }}+\frac{K_{2}}{n^{\rho
\frac{r}{2}}}+\frac{K_{3}}{n^{\rho \frac{\left( \beta +1\right) p}{\beta +p}%
}\left( \ln n\right) ^{\frac{\left( \beta +1\right) p}{\beta +p}}}=0.
\end{equation}%
Since there exists a natural integer $n_{\sigma }$ such that%
\begin{equation}
\forall \ n\in \mathbb{N},n\geq n_{\sigma }-1\Longrightarrow \frac{K_{1}}{%
n^{1+\delta }}+\frac{K_{2}}{n^{\rho \frac{r}{2}}}+\frac{K_{3}}{n^{\rho \frac{%
\left( \beta +1\right) p}{\beta +p}}\left( \ln n\right) ^{\frac{\left( \beta
+1\right) p}{\beta +p}}}\leq \sigma ,  \label{fin}
\end{equation}%
thus, (\ref{ic}) arises from (\ref{dering}) and (\ref{fin}).
\end{proof}

\noindent 4. \textbf{Numerical results}

In this section, a simulation study is proposed to check the validity of our obtained theoretical results. We consider two examples. In the first one, a contractive function where its unique fixed point is known exactly and we compare the fixed point to the approximated ones obtained using the Mann's iterative algorithm. In the second example, we consider a classical problem from astronomy, where the mathematical equation cannot be solved to obtain the exact value of the fixed point and we use the Cauchy's criterium to compare two successive iterates to insure the convergence of the sequence obtained using iterative Mann's algorithm.

\begin{equation*}
x_{n+1}=\left( 1-\frac{a}{n}\right) x_{n}+\frac{a}{n}\left[ f\left(
x_{n}\right) +\frac{1}{n}\xi _{n}\right]
\end{equation*}

$0<a(1-c)<1,$ $\xi _{0}=0,$ $n\in
\mathbb{N}
^{\ast }.$

To characterize the strong mixing random errors $(\xi _{i})$, we consider an
autoregressive model $(\xi_{i})_{i}$ of order 1 (see \cite{dahm}) described as follows
\begin{equation}
\begin{array}{c}
\xi _{i+1}=\varphi \xi _{i}+g_{i},%
\end{array}
\label{regress}
\end{equation}%
where $g_{i}$ is a Gaussian white noise process, $\varphi $ is a constant
such that $\left\vert \varphi \right\vert <1$. For the simulation of Gaussian random variables $(g_{i})_{i}$, we use the method of Box-Muller :%
\begin{equation}
\begin{array}{cc}
g_{k} & =\sqrt{-2ln(u_{1})}\ cos(2\pi u_{2})
\end{array}%
  \label{gauss}
\end{equation}%
where $u_{1}$ and $u_{2}$ are uniform distributed random numbers.
\begin{example}
We consider the following function defined by:
\begin{equation*}
\begin{array}{cc}
f & :\left[ 0,5\right] \rightarrow \left[ 0,5\right]  \\
x & \mapsto \sqrt{x+1}%
\end{array}%
\end{equation*}%
The function $f$ is a contractive function with $c=$ $%
\underset{x\in \left[ 0,+\infty \right) }{max}\left\vert f^{'}
(x)\right\vert =\frac{1}{2}.$ Hence $f$ has a unique fixed point $x^{\ast
}=\frac{1+\sqrt{5}}{2}=1,618033988749895$, which is known as golden number. For $x_{1}=1.3,a=\frac{1}{4}$ and
$\varphi =0.8,$ the following results are obtained:%
\begin{equation*}
\begin{tabular}{|l|l|l|}
\hline
$n$ & $x_{n}$ & $\left\vert x_{n}-x^{\ast }\right\vert $ \\ \hline
$10^{3}$ & 1.614142671526978 & 0.003891317222917 \\ \hline
$10^{4}$ & 1.615420224332314 & 0.002613764417581 \\ \hline
$10^{5}$ & 1.616115916146472 & 0.001918072603423 \\ \hline
\end{tabular}%
\end{equation*}
\end{example}

\begin{example}\label{astronomy_example} Most of mathematical problems come from other  engineering sciences ( physics,
chemistry, geology, astronomy, etc.). When studying some physical problems
using the appropriate mathematical models, we obtain an equation or a set of
equations and usually cannot be solved analytically because in general, these
equations are corrupted by noise or the known mathematical tools do not
allow us to solve them. To illustrate this fact, we consider the following \
classical example from astronomy.

Consider a planet in an orbit around the sun as described by the following
diagram.
\begin{figure}[ht]
  \
  \centering
  \includegraphics[scale=0.4]{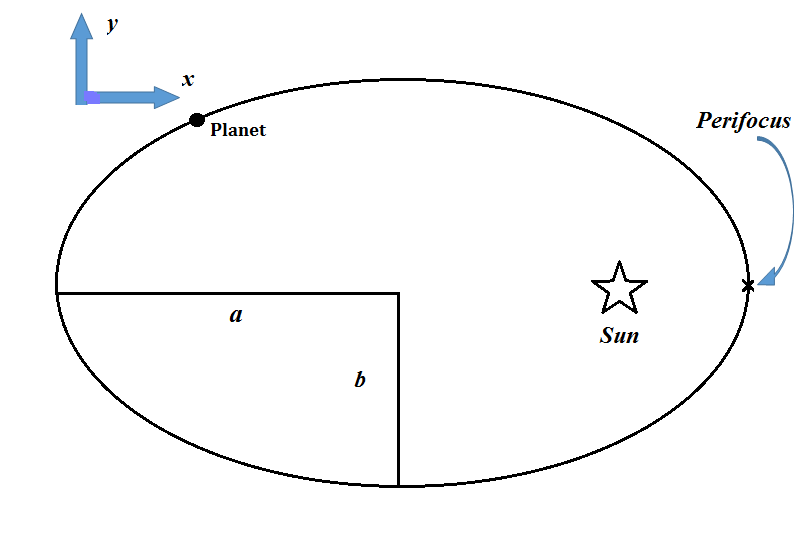}\\
\end{figure}

Let $n$ be the mean angular motion of the Mercury's orbit around the sun, $t$ the elapsed time since the planet
was last closet to the sun (this is called perifocus or perihelion in astronomy) and $e=%
\sqrt{1-\frac{b^{2}}{a^{2}}}$, the eccentricity of the planet's elliptical
orbit.\ Using Kepler's laws of planetary motion, we obtain the location of
the planet at time $t.$%
\[
\left\{
\begin{array}{c}
x=a\left( \cos \left( E\right) -e\right)  \\
y=a\sqrt{1-e^{2}}\sin \left( E\right)%
\end{array}%
\right.
\]%
The quantity $E$ is called the eccentric anomaly and is given by the following
equation%
\[
E=nt+e\sin \left( E\right) =M+e\sin \left( E\right).
\]%
where $M$ is called the mean anomaly which increases linearly in time at the rate $n$. Note that $E$ is the fixed point of the function $f,\ $where $f\left(
x\right) =M+e\sin \left( x\right) $ for a given time $t$ and the frequency of the orbit $\omega$. In this
equation, we cannot find an explicit formula of the eccentric anomaly $E.$
It is easy to check that $f$ is a contraction, moreover, we have
\[
\left\vert f\left( x\right) -f\left( y\right) \right\vert \leq e\left\vert
x-y\right\vert
\]%
which ensures the existence and uniqueness of the fixed point $E.$

For our simulation, by choosing the planet Mercury, we have its eccentricity $e= 0.20563069$ and the mean anomaly M= $3.05076572$ 
(The Mann's process is implemented for $a=0.9,\varphi =0.7),$ and given an initial guess $x_{1}=3$, we obtain the
following iterates :
\begin{equation*}
\begin{tabular}{|l|l|l|l|}
\hline
$n$ & $x_{n}$ & $\left\vert x_{n}-x_{n-1}\right\vert $&$\left\vert x_{n}-x_{fp}\right\vert $ \\ \hline
 100&  3.066277803444744&  5.084582991976561e-06&3.292480386907215e-05 \\ \hline
1000 & 3.066247563732222 & 2.842158499660741e-07&2.685091347043311e-06 \\ \hline
$10^{4}$ &  3.066245125153754 &  6.291136500635730e-08&2.465128789985727e-07 \\ \hline
$10^{5}$ & 3.066244900326525 &  7.696832948766996e-09&2.168565016447133e-08 \\ \hline
\end{tabular}%
\end{equation*}%

\end{example}

\begin{remark}
Note that the numerical solution of the equation $f(x)=x$ given by Matlab is $x_{fp}=3.066244878640875$.
As we can observe, the used Mann algorithm gives nice approximations of the unique fixed point of the function $f$.
Thus, the complementary numerical examples considered above make the obtained theoretical results of convergence well palpable.
\end{remark}


\begin{thebibliography}{99}


\bibitem{Ait} Ait Saidi, A.,\ Dahmani, A. $\left( 2012\right) .$ Convergence
de l'algorithme de Kiefer-Wolfowitz dans un cadre m\'{e}langeant, Ann. Sci.
Math. Qu\'{e}bec. 36:$1-10.$


\bibitem{banac} Banach, S. (1922). Sur les op\'{e}rations dans les ensembles
abstraits et leur applications aux \'{e}quations int\'{e}grales, Fund. Math.
3:133--181.

\bibitem{Ber} Berinde, V.\ ($2007).$ Iterative Approximation of Fixed
Points, Lecture Notes in Mathematics, Springer$.$

\bibitem{Ceg} Cegielski, A. $\left( 2012\right) .$ Iterative Methods for
Fixed Point Problems in Hilbert Spaces, Springer$.$

\bibitem{cha2} Chang, S.S.,\ Kim, J.K. $(2003).$ Convergence Theorems of the
Ishikawa Type Iterative Sequences with Errors for Generalized
Quasi-Contractive Mappings in Convex Metric Spaces, Applied Mathematics
Letters, 16:535-542.


\bibitem{dahm} Dahmani, A.,\ Ait Saidi, A. $\left( 2010\right)$. Consistency
of Robbins Moro's algotithm within a mixing framework, Bulletin de la Soci%
\'{e}t\'{e} Royale des Sciences de Li\`{e}ge, 79:131-140.

\bibitem{Dou} Doukhan, P. $\left( 1994\right)$. Mixing: Properties and
examples, Lecture Notes in Statistics (Vol. 85), Springer.


\bibitem{Fer} Ferraty, F., Goia, A. and Vieu, P. $\left( 2002\right)$. Functional nonparametric model for time series: a fractal approach for
dimension reduction, Test 11, 2:317-344.

\bibitem{Hus} Hussaina, N., Narwalb,S., R. Chughc, R. and Kumard, V.($2016$%
). On convergence of random iterative schemes with errors for strongly
pseudo-contractive Lipschitzian maps in real Banach spaces, J. Nonlinear
Sci. Appl. 9 :157-3168.

\bibitem{Ibr} Ibragimov, I.A., Yu.V.\ Linnik, Yu.V.$\left( 1971\right)$. Independent and stationary sequences of random variables, Berlin, Springer$.$

\bibitem{Ish} Ishikawa, S. $\left( 1974\right)$. Fixed points by a new
iteration method, Proc. Amer. Math. Soc., 44:147-150.

\bibitem{Kan} Kang, S.M., Ali, F., Arif, R., Kwun, Y. C. and Jabeen, S. $%
\left( 2016\right)$. On the Convergence of Mann and Ishikawa Type Iterations
in the class of Quasi Contractive Operators, Journal of Computational
Analysis \& Applications. 21:451-459.

\bibitem{Kim} Kim, G.E., Kim, T.H. $(2001)$. Mann and Ishikawa Iterations
with Errors for Non-Lipschitzian Mappings in Banach Spaces, Computers and
Mathematics with Applications 42:1565-1570$.$
%

\bibitem{Liu3} Liu, Z., Kang, S.M. ($2001$). Stability of Ishikawa Iteration
Methods with Errors for Strong Pseudocontractions and Nonlinear Equations
Involving Accretive Operators in Arbitrary Real Banach Spaces. Mathematical
and Computer Modelling, 34:319-330.


\bibitem{Man} Mann, W.R. $(1953)$. Mean value methods in iteration. Proc.
Amer. Math. Soc., 4:506-510.



\bibitem{Rio} Rio, E.($2000)$. Th\'{e}orie asymptotique des processus al\'{e}%
atoires faiblement d\'{e}pendants, Math\'{e}matiques and Applications,
(Vol.31). Berlin:Springer-Verlag.

\bibitem{Ros} Rosenblatt, M.$\left( 1956\right)$. A central limit theorem
and a strong mixing condition, Proc. Nat. Acad. Sci., 42:43-47, USA.


\end{thebibliography}
\end{document}